\documentclass[11pt]{article}

\usepackage{CJK}
\usepackage{amsmath}
\usepackage{bm}
\usepackage{amsfonts}
\usepackage{latexsym}
\usepackage{amssymb}
\usepackage{stmaryrd}
\usepackage{overpic}
\usepackage{graphicx}
\usepackage{multirow}
\usepackage{mathrsfs}
\usepackage{float}
\usepackage{color}
\usepackage{setspace}
\usepackage{titling}
\usepackage{subfigure}
\usepackage{dsfont}
\usepackage{titlesec}

\newtheorem{theorem}{Theorem}[section]

\newtheorem{lemma}[theorem]{Lemma}

\newtheorem{remark}[theorem]{Remark}
\newenvironment{proof}[1][Proof]{\textbf{#1.} }
{\ \rule{0.75em}{0.75em}\smallskip}

\title{A new finite element approach for the Dirichlet eigenvalue Problem}
\author{
Wenqiang Xiao$^1$, Bo Gong$^1$, Jiguang Sun$^2$, and Zhimin Zhang$^{1,3}$\\
\small{$^1$ Beijing Computational Science Research Center, Beijing 100193, China.}\\
\small{$^2$ Department of Mathematical Sciences, Michigan Technological University, Houghton, MI 49931, USA.}\\
\small{$^3$ Department of Mathematics, Wayne State University, Detroit, MI 48202, USA.}}
\date{}
\begin{document}
\maketitle

\begin{abstract}
In this paper, we propose a new finite element approach, which is different than the classic Babu\v{s}ka-Osborn theory, to approximate Dirichlet eigenvalues. The Dirichlet eigenvalue problem is formulated as the eigenvalue problem of a holomorphic Fredholm operator function of index zero. Using conforming finite elements, the convergence is proved using the abstract approximation theory for holomorphic operator functions. The spectral indicator method is employed to compute the eigenvalues. A numerical example is presented to validate the theory.
\end{abstract}


\section{Introduction}
Finite element methods for eigenvalue problems have been studied extensively \cite{babuska1991, Boffi2010AN, sun2016}.
In this paper, we propose a new finite element
approach for the Dirichlet eigenvalue problem. The problem is formulated as an eigenvalue problem of a holomorphic Fredholm operator function \cite{Gohberg2009}.
Using Lagrange finite elements, the convergence is proved by the abstract approximation theory for holomorphic operator functions \cite{karma1996I, karma1996II}.

The new approach has the following characteristics: 1) it provides a new finite element methodology which is different than the classic
Babu\v{s}ka-Osborn theory; 2) it can be applied to a large class of nonlinear eigenvalue problems \cite{Gong2019}; and
3) combined with the spectral indicator method \cite{huang2016, huang2018, Gong2019}, it can be parallelized to compute many eigenvalues effectively.

The rest of the paper is arranged as follows.
In Section 2, preliminaries of holomorphic Fredholm operator functions and the associated abstract approximation are presented.
In Section 3, we reformulate the Dirichlet eigenvalue problem as the eigenvalue problem of a holomorphic Fredholm operator function of index zero.
The linear Lagrange finite element is used for discretization and the convergence is proved using the abstract approximation results of Karma  \cite{karma1996I, karma1996II}.
In Section 4, the spectral indicator method \cite{huang2016, huang2018, Gong2019} is applied to compute the eigenvalues of the unit square.


\section{Preliminaries.}
We present some preliminaries on the eigenvalue approximation theory of holomorphic Fredholm operator functions following \cite{karma1996I,karma1996II}.
Let $X,Y$ be complex Banach spaces, $\Omega\subset\mathbb{C}$ be compact.
Denote by $\mathcal{L}(X,Y)$ the space of bounded linear operators
and $\Phi_0(\Omega,\mathcal{L}(X,Y))$ the set of holomorphic Fredholm operator functions of index zero \cite{Gohberg2009}.
Assume that $F \in \Phi_0(\Omega,\mathcal{L}(X,Y))$.
We consider the problem of finding $(\lambda,u)\in \Omega\times X,~u\neq 0$, such that
\begin{equation}\label{nonlinear eig}
F(\lambda)u=0.
\end{equation}
The resolvent set and the spectrum of $F$ are defined as
\[
\rho(F)=\{\lambda\in \Omega:~F(\lambda)^{-1} \in \mathcal{L}(Y,X) \},\quad \sigma(F)=\Omega\backslash\rho(F).
\]
Throughout the paper, we assume that $\rho(F)\neq\emptyset$. Then the spectrum $\sigma(F)$ has no cluster points in $\Omega$,
and every $\lambda \in \sigma(F)$ is an eigenvalue \cite{karma1996I}.

To approximate the eigenvalues of $F$, we consider a sequence of operator functions 
$F_n\in \Phi_0(\Omega,\mathcal{L}(X_n,Y_n)), n\in\mathbb{N}$. Assume the following properties hold.
\begin{itemize}
\item[(b1)] There exist Banach spaces $X_n,Y_n$, $n\in\mathbb{N}$ and linear bounded mappings $p_n\in\mathcal{L}(X,X_n), q_n\in\mathcal{L}(Y,Y_n)$  with the property
\begin{equation}\label{pnqn}
\lim\limits_{n \rightarrow\infty}\|p_n v\|_{X_n}=\|v\|_X,\, v\in X,\quad
\lim\limits_{n \rightarrow\infty}\|q_n v\|_{Y_n}=\|v\|_Y,\, v\in Y.
\end{equation}
\item[(b2)] $\{F_n(\cdot)\}_{n \in \mathbb{N}}$ is equibounded on $\Omega$, i.e., there exists a constant $c$ such that
\[
\|F_n(\lambda)\| \le c \quad \forall \lambda \in \Omega, n \in \mathbb{N}.
\]
\item[(b3)] $\{F_n(\cdot)\}_{n \in \mathbb{N}}$ approximates $F(\lambda)$ for every $\lambda\in\Omega$, i.e., 
\[
\lim_{n \to \infty} \|F_n(\lambda)p_n u - q_n F(\lambda) u\|_{Y_n} = 0 \quad \forall u \in X. 
\]
\item[(b4)] $\{F_n(\cdot)\}_{n \in \mathbb{N}}$ is regular for every $\lambda \in \Omega$, i.e., if $\|x_n\| \le 1\, (n \in \mathbb N)$ and
$\{F_n(\lambda)x_n\}_{n \in \mathbb N}$ is compact in the sense of Karma \cite{karma1996I}, then
$
\{x_n\}_{n \in \mathbb N}  \text{ is compact}.
$
\end{itemize}

\section{Finite Element Approximation}
Let $D\subset \mathbb{R}^2$ be a bounded Lipschitz domain.
The Dirichlet eigenvalue problem is to find $\lambda\in\mathbb{C}$ and $u\neq 0$ such that
\begin{equation}\label{possioneig}
-\triangle u =\lambda u ~~\text{in}~D \quad \text{and} \quad
u =0~~\text{on}~\partial D.
\end{equation}
The associated source problem is, given $f$, to find $u$ such that
\begin{equation}\label{sourcepro}
-\triangle u =f ~~{\rm in}~D  \quad \text{and} \quad
u =0~~{\rm{on}}~\partial D.
\end{equation}
For $f\in L^2(D)$, the weak formulation of \eqref{sourcepro} is finding $u\in H^1_0(D)$ such that
\begin{equation}\label{PriPro}
a(u,v)=(f,v)\qquad \forall~v\in H^1_0(D),
\end{equation}
where
$$a(u,v)=\int_D \nabla u\cdot\nabla v~dx,\quad (f,v)=\int_D fv~dx.$$

Due to the wellposeness of \eqref{PriPro} (see, e.g., \cite{sun2016}), there exists a linear compact solution operator
$T: L^2(D)\rightarrow H^1_0(D) \subset L^2(D)$ such that $Tf=u$.
The Dirichelt eigenvalue problem \eqref{possioneig} is equivalent to the operator eigenvalue problem: $T(\lambda u)=u$.

Assume that $0 \notin \Omega$. Define a nonlinear operator function $F: \Omega \to \mathcal{L}(L^2(D), L^2(D))$ by
\begin{equation}\label{Flambda}
F(\lambda):=T-\frac{1}{\lambda}I, \quad \lambda \in \Omega,
\end{equation}
where $I$ is the identity operator.
Clearly, $\lambda$ is a Dirichlet eigenvalue of \eqref{possioneig} if and only if $\lambda$ is an eigenvalue of $F(\lambda)$.
\begin{lemma} Let $\Omega\subset \mathbb{C}\backslash\{0\}$ be a compact set. Then
$F(\cdot):\Omega\rightarrow \mathcal{L}(L^2(D),L^2(D))$ is a holomorphic Fredholm operator function of index zero.
\end{lemma}
\begin{proof}
Since $T$ is compact and $I$ is the identity operator, $F(\lambda), \lambda \in \Omega$ is a
Fredholm operator of index zero. Clearly, $F(\lambda)$ is holomorphic in $\Omega$.
\end{proof}

In the rest of the paper, $\|\cdot\|$ stands for $\|\cdot\|_{L^2(D)}$ and $C>0$ is a generic constant.
Let $\mathcal{T}_h$ be a regular triangular mesh for $D$ with mesh size $h$ and
$V_h\subset H^1_0(D)$ is the linear Lagrange element space associated with $\mathcal{T}_h$. Then the discrete formulation of \eqref{PriPro} is to find $u_h\in V_h$, such that
\begin{equation}\label{FemPro}
a(u_h,v_h)=(f,v_h)=(p_hf,v_h)\qquad \forall~v_h\in V_h,
\end{equation}
where $p_h:L^2(D)\rightarrow V_h$ is the $L^2$-projection operator.
Obviously, $p_h$ is bounded and $\|p_hf-f\|\rightarrow 0$ as $h\rightarrow 0$, for any $f\in L^2(D)$. Let $f_h=p_hf$. 
The well-posedness of \eqref{FemPro} implies that $\|u_h\|\leq C\|f_h\|$.

Let $u$ and $u_h$ be the solutions of \eqref{PriPro} and \eqref{FemPro}, respectively.
Then there exists $\beta > 1/2$ ($\beta = 1$ if $D$ is convex) such that (see Sec. 3.2 in \cite{sun2016})
\begin{equation}\label{L2 error}
\|u-u_h\|\leq Ch^{2\beta}\|f\|.
\end{equation}

Let $T_h : V_h \rightarrow V_h$, $T_h f_h = u_h$ be the solution operator of \eqref{FemPro}.
Define an operator function $F_h: \Omega \to \mathcal{L}(V_h, V_h)$
\begin{equation}\label{Fhlambda}
F_h(\lambda):=T_h-\frac{1}{\lambda}I.
\end{equation}
The error estimate (3.8) indicates that $\Vert T - T_h p_h\Vert \leqslant Ch^{2\beta}$. This implies that
\begin{align}
\Vert F(\lambda)|_{V_h} - F_h(\lambda)\Vert \leqslant Ch^{2\beta},
\label{F-F_h}
\end{align}
which is due to
$
\Vert F(\lambda) v_h - F_h(\lambda) v_h\Vert = \Vert T v_h - T_h v_h \Vert \leqslant \Vert T - T_h p_h\Vert \Vert v_h\Vert,
$
for all $v_h\in V_h$.

\begin{lemma} There exists $h_0>0$ small enough such that, for every compact set $\Omega\subset \mathbb{C}\backslash\{0\},$
\begin{equation}\label{bounded}
\sup\limits_{h}\sup\limits_{\lambda\in \Omega}\|F_h(\lambda)\|<\infty,\quad h<h_0.
\end{equation}
\end{lemma}
\begin{proof} Let $h$ be sufficiently small and $f_h\in V_h$. Then
\[
\|F_h(\lambda)f_h\|= \left \|u_h-\frac{1}{\lambda}f_h\right \|
\leq \|u_h\|+\frac{1}{|\lambda|}\|f_h\|
\leq \left(C+\frac{1}{|\lambda|}\right)\|f_h\|.
\]
Since $\Omega$ is compact and $0 \notin \Omega$, \eqref{bounded} holds.
\end{proof}
\begin{lemma}
Let $f\in L^2(D)$. Then
\begin{equation}\label{consistent}
\lim\limits_{h\rightarrow 0}\|F_h(\lambda)p_hf-p_hF(\lambda)f\|=0.
\end{equation}
\end{lemma}
\begin{proof}
From \eqref{L2 error}, we have that
\begin{align*}
\|F_h(\lambda)p_hf-p_hF(\lambda)f\|
&= \left\|T_h(\lambda)p_hf-\frac{1}{\lambda}p_hf-p_hT(\lambda)f+\frac{1}{\lambda}p_hf \right\|\\
&=\|u_h-u+u-p_hu\|\leq\|u-u_h\|+\|u-p_hu\|\leq Ch^{2\beta}\|f\|.
\end{align*}
\end{proof}

Now we are ready to present the main convergence theorem.
\begin{theorem}
Let $\lambda_0\in \sigma(F)$. Assume that $h$ is small enough. Then there exists $\lambda_h\in\sigma(F_h)$
such that $\lambda_h\rightarrow\lambda_0$ as $h\rightarrow 0$. For any sequence $\lambda_h\in\sigma(F_h)$ 
the following estimate hold
\begin{equation}
\label{EVorder} |\lambda_h-\lambda_0|\leq Ch^{\frac{2\beta}{r_0}},
\end{equation}
where $r_0$ is the maximum rank of eigenvectors.
\end{theorem}
\begin{proof}
Let $\{h_n\}$ be a sequence of sufficiently small positive numbers with
$h_n\rightarrow 0$ as $n\rightarrow\infty$ and $F_n(\lambda):=F_{h_n}(\lambda)$, $V_n:=V_{h_n}$ and $p_n:=p_{h_n}$.
Clearly, (b1) holds with $X = Y = L^2(D)$, $X_n = Y_n = V_n$, and $q_n = p_n$. (b2) and (b3) hold due to Lemma 3.2 and 3.3.

To verify (b4), assume that $v_n \in V_n$, $n\in \mathbb{N}^{\prime}\subset \mathbb{N}$  with $\Vert v_n\Vert  \le 1$ and
\begin{equation}\label{assumption3}
\lim\limits_{n\rightarrow\infty}\|F_n(\lambda)v_n-p_ny\|=0,
\end{equation}
for some $y\in L^2(D)$. We estimate $\Vert v_n - p_n v\Vert$ as follows by considering $\lambda\in \rho(F)$ and $\lambda\in \sigma(F)$ separately.

If $\lambda\in \rho(F)$, then $F(\lambda)^{-1}$ exists and is bounded. Let $v = F(\lambda)^{-1}y$.
We have
\begin{align*}
v_n - p_n v = F(\lambda)^{-1}\big((F(\lambda)-F_n(\lambda))(v_n - p_nv) + F_n(\lambda)v_n - p_n F(\lambda) v + p_n F(\lambda) v- F_n(\lambda) p_n v\big).
\end{align*}
Recalling $\Vert F(\lambda)|_{V_n} - F_n(\lambda)\Vert\leqslant Ch_n^{2\beta}$ from \eqref{F-F_h} it holds
\begin{align}
\Vert v_n - p_n v\Vert \leqslant C \big(h_n^{2\beta} \Vert v_n - p_n v\Vert + \Vert F_n(\lambda)v_n - p_n F(\lambda)v\Vert + \Vert p_n F(\lambda)v - F_n(\lambda)p_n v\Vert\big).
\label{v_N}
\end{align}
Using \eqref{assumption3} and Lemma 3.3 we have $\Vert v_n - p_n v\Vert\rightarrow 0$ as $n\rightarrow \infty$.

Assume that $\lambda\in \sigma(F)$. Let $G(\lambda)$ be the finite dimensional generalized eigenspace of $\lambda$ \cite{karma1996I}.
We denote by $P_{G(\lambda)}$ the projection from $L^2(D)$ to $G(\lambda)$, by $F(\lambda)^{-1}$ the inverse of $F(\lambda)|_{L^2(D)/G(\lambda)}$
from $\mathcal{R}(F(\lambda))$ to ${L^2(D)/G(\lambda)}$. Due  \eqref{assumption3}, we have that 
\[
\Vert F(\lambda) v_n - y\Vert \leqslant \Vert F(\lambda)v_n - F_n(\lambda)v_n\Vert + \Vert F_n(\lambda)v_n - p_n y\Vert + \Vert p_n y-y\Vert\rightarrow 0, \quad n \to \infty.
\]
Since $\mathcal{R}(F(\lambda))$ is closed, $y \in \mathcal{R}(F(\lambda))$.
Let  $v^{\prime} := F(\lambda)^{-1} y$ and $v_n^{\prime}:= (I-p_n P_{G(\lambda)})v_n$. Since
\[
\Vert F_n(\lambda)p_n P_{G(\lambda)} v_n\Vert \leqslant \Vert F_n(\lambda) p_n - p_n F(\lambda)\Vert \Vert P_{G(\lambda)}v_n\Vert \rightarrow 0, \quad n \to \infty, 
\] 
by Lemma 3.3, similar to \eqref{v_N}, we deduce that
\begin{align*}
\Vert v_n^{\prime} - p_n v^{\prime}\Vert \leqslant (1-Ch_{n}^{2\beta})^{-1}C\big(\Vert F_n(\lambda)v_n^{\prime} - p_n F(\lambda) v^{\prime}\Vert + \Vert p_n F(\lambda)v^{\prime} - F_n(\lambda)p_n v^{\prime}\Vert\big) \rightarrow 0.
\end{align*}
On the other hand, since $G(\lambda)$ is finite dimensional, there is a subsequence $\mathbb{N}^{\prime\prime}$ and 
$v^{\prime\prime}\in G(\lambda)$ such that $\Vert P_{G(\lambda)}v_n - v^{\prime\prime}\Vert\rightarrow 0$ as $\mathbb{N}^{\prime\prime} \ni n \rightarrow \infty$. Therefore we have
\begin{align*}
\Vert v_n - p_n v\Vert \leqslant \Vert v_n^{\prime} - p_n v^{\prime}\Vert + \Vert p_n P_{G(\lambda)}v_n - p_n v^{\prime\prime}\Vert \rightarrow 0,\quad \text{as}\quad \mathbb{N}^{\prime\prime} \ni n\rightarrow \infty,
\end{align*}
where $v := v^{\prime} + v^{\prime\prime}$.
Now, we have verified (b1)-(b4) which are the conditions for Theorem 2 of \cite{karma1996II}. Then \eqref{EVorder} follows readily.
\end{proof}
\begin{remark}
Under the same conditions of the above theorem, it is possible to obtain error estimates for generalized eigenvectors \cite{karma1996II}, 
which is not included in this paper due to simplicity.
\end{remark}

\section{Numerical Results}
Let $D$ be the unit square $(0,1)\times (0,1)$.
The smallest eigenvalue is $2\pi^2$ and its rank $r_0=1$. We use the linear Lagrange element on a series of uniformly refined meshes for discretization.
The spectral indicator method \cite{huang2016, huang2018, Gong2019} is employed to compute the smallest eigenvalue of \eqref{Fhlambda}.
The results are shown in Table~\ref{table1}, which confirms the second order convergence.
\begin{table}[h]
\centering
\begin{tabular}{cccc}
\hline
$h$ &$\lambda_h$   &$|\lambda-\lambda_h|$  &convergence order\\
\hline
1/10    &19.9281  &0.1889   &- \\
1/20    &19.7871  &0.0479   &1.9795 \\
1/40    &19.7512  &0.0120   &1.9970 \\
1/80    &19.7422  &0.0029   &2.0489 \\
\hline
\end{tabular}
\caption{The smallest Dirichlet eigenvalue of the unit square~(linear Lagrange element)}
\label{table1}
\end{table}

\section*{Aknowledgement}
The research of B. Gong is supported partially by China Postdoctoral Science Foundation Grant 2019M650460.
The research of J. Sun is supported partially by MTU REF.
The research of Z. Zhang is supported partially by the National Natural Science Foundation of China grants NSFC 11871092, NSAF U1930402, and NSF 11926356.

\end{document}